\documentclass[1p,final]{elsarticle}
\usepackage{amsfonts,natbib,pslatex}

\usepackage{latexsym,amsmath,amsthm,amssymb,amsxtra,mathrsfs,times,CJK}

\usepackage{color}
\usepackage{amsfonts,color}
\usepackage{amssymb,amsmath,latexsym}
\usepackage{amssymb}

\usepackage[para]{threeparttable}

\newcommand{\gf}{{\mathbb{GF}}}

\newcommand{\C}{{\mathcal C}}

\newtheorem{theorem}{Theorem}[section]
\newtheorem{lemma}[theorem]{Lemma}

\newtheorem{open}[theorem]{Open Problem}

\journal{Finite Fields and Their Applications}

\begin{document}

\begin{frontmatter}



\title{On an open problem about a class of optimal ternary cyclic codes}
\author[SWJTU]{Dongchun Han}
 \ead{han-qingfeng@163.com}
\author[SWJTU]{Haode Yan\corref{cor1}}
 \ead{hdyan@swjtu.edu.cn}

 \cortext[cor1]{Corresponding author}
 \address[SWJTU]{School of Mathematics, Southwest Jiaotong University, Chengdu, 610031, China}
%\address[NNU]{College of Mathematics and Statistics, Northwest Normal University, Lanzhou, 730070, China}
\tnotetext[fn1]{D. Han's research was supported by the National Science Foundation of China Grant No.11601448 and the Fundamental Research Funds for the Central Universities of China under Grant 2682016CX121. H. Yan's research was supported by the National Natural Science Foundation of
China Grant No.11801468 and the Fundamental Research Funds for the Central Universities of China under Grant 2682018CX61.}

% \author[label1,label2]{<author name>}
% \address[label1]{<address>}
% \address[label2]{<address>}

%\author[SWJTU]{Cuiling Fan\corref{cor1}}
% \ead{cuilingfan@163.com}
%\author[UB]{Nian Li}
% \ead{nianli.2010@gmail.com}
%\author[SWJTU]{Zhengchun Zhou}
% \ead{zzc@home.swjtu.edu.cn,zczhou@126.com}
%

% \address[SWJTU]{School of Mathematics, Southwest Jiaotong University, Chengdu, 610031, China}
%\address[UB]{Department of Informatics, University of Bergen, N-5020 Bergen, Norway}
%Department of Informatics, University of Bergen, N-5020 Bergen, Norway
\begin{abstract}
Cyclic codes are a subclass of linear codes and have applications in consumer electronics, data storage systems and communication systems as they have efficient encoding and decoding algorithms. In this paper, we settle an open problem about a class of optimal ternary cyclic codes which was proposed by Ding and Helleseth \cite{Ding-Heelseth}. Let $\C_{(1,e)}$ be a cyclic code of length $3^m-1$ over $\gf(3)$ with two nonzeros $\alpha$ and $\alpha^e$, where $\alpha$ is a generator of $\gf(3^m)^*$ and $e$ is a given integer. It is shown that $\C_{(1,e)}$ is optimal with parameters $[3^m-1,3^m-1-2m,4]$ if one of the following conditions is met. 1) $m\equiv0(\mathrm{mod}~  4)$, $m\geq 4$, and $e=3^\frac{m}{2}+5$. 2) $m\equiv2(\mathrm{mod}~  4)$, $m\geq 6$, and $e=3^\frac{m+2}{2}+5$.

\end{abstract}


\end{frontmatter}

\begin{abstract}
Cyclic codes are a subclass of linear codes and have applications in consumer electronics,
data storage systems, and communication systems as they have efficient encoding and
decoding algorithms. In this paper, a conjecture proposed by Ding and Helleseth in 2013 about a class of optimal ternary cyclic codes $\C_{(1,e)}$ for $e=3^h+5, 2\leq h\leq m-1$ with parameters $[3^m-1,3^m-1-2m,4]$ is settled, if one of the following conditions is met:
\begin{enumerate}
\item $m\equiv0(\mod 4)$, $m\geq 4$, and $h=\frac{m}{2}$.
\item $m\equiv2(\mod 4)$, $m\geq 6$, and $h=\frac{m+2}{2}$.
\end{enumerate}
\end{abstract}

%\begin{keywords}
%Cyclic codes, weight distribution, quadratic form, sphere packing bound.
%\end{keywords}

\section{Introduction}

Cyclic codes are an important subclass of linear codes and have been extensively studied \cite{Klove}. Let $p$ be a prime, $m$ be a positive integer. Let $\gf(p)$ and $\gf(p^m)$ denote the finite fields with $p$ and $p^m$ elements, respectively. A linear $[n,k,d]$ code $\C$ over the finite field $\gf(p)$ is a $k$-dimensional subspace of $\gf(p)^n$ with minimum Hamming distance $d$, and is called {cyclic} if any cyclic shift of a codeword
is another codeword of $\C$. Let $\gcd(n,p)=1.$ By identifying any vector $(c_0,c_1,\cdots,c_{n-1})\in \gf(p)^n$ with
 \[c_0+c_1x+c_2x^2+\cdots+c_{n-1}x^{n-1}\in\gf(p)[x]/(x^n-1),\]
 any cyclic code of length $n$ over $\gf(p)$ corresponds to an ideal of the polynomial residue class ring $\gf(p)[x]/(x^n-1)$.
 It is well known that every ideal of $\gf(p)[x]/(x^n-1)$ is principal. Any cyclic code $\C$ can be expressed as $\C=\langle g(x) \rangle$, where $g(x)$ is monic and has the least degree. Then $g(x)$ is called the {generator polynomial} and $h(x)=(x^n-1)/g(x)$ is referred to as the { parity-check polynomial} of $\C$. For some recent developments of cyclic codes, the readers are referred to \cite{CDY05}, \cite{Ding12}-\cite{Ding-Heelseth}, \cite{FLZ}, \cite{TFeng}, \cite{Li-CJ}-\cite{LiN},  \cite{Schmidt}-\cite{YangJing}, \cite{ZhengDB}-\cite{ZhouDingTC}  and the references therein.

Let $\alpha$ be a generator of $\gf(3^m)^*=\gf(3^m)\setminus\{0\}$ and $m_{i}(x)$ be
the minimal polynomial of $\alpha^i$ over $\gf(3)$, where $1\leq i\leq 3^m-1$.
Let $\C_{(1,e)}$ be the cyclic code over $\gf(3)$  with generator polynomial $m_1(x)m_e(x)$, where $e$ is an integer such that
$\alpha$ and $\alpha^e$ are nonconjugate. Carlet, Ding and Yuan \cite{CDY05} proved that $\C_{(1,e)}$ has parameters $[3^m-1,3^m-1-2m,4]$ when $x^e$ are certain perfect nonlinear monomials over $\gf(3^m)$. Notice that the ternary cyclic code with parameters $[3^m-1,3^m-1-2m,4]$ is optimal according to the Sphere Packing bound. In 2013, Ding and Helleseth \cite{Ding-Heelseth} constructed several classes of
optimal ternary cyclic codes $\C_{(1,e)}$ with parameters $[3^m-1,3^m-1-2m,4]$ by employing some monomials $x^e$ over $\gf(3^m)$ including almost perfect nonlinear monomials. In addition, nine open problems about $\C_{(1,e)}$ with parameters $[3^m-1,3^m-1-2m,4]$ were proposed in \cite{Ding-Heelseth}. Recently, two of the nine open problems were solved, see \cite{LiN, LZH}. Moreover, an open problem proposed in \cite{Ding-Heelseth} is shown as follows.

\begin{open}[Open Problem 7.12, \cite{Ding-Heelseth}]\label{open-1}
Let $e=3^h+5$, where $2\leq h\leq m-1$. Let $m$ be even. Is it true that the ternary cyclic code $\C_{(1,e)}$ has parameters  $[3^m-1,3^m-1-2m,4]$ if one of the following conditions is met?
\begin{enumerate}
\item $m\equiv0(\mathrm{mod}~  4)$, $m\geq 4$, and $h=\frac{m}{2}$.

\item $m\equiv2(\mathrm{mod}~  4)$, $m\geq 6$, and $h=\frac{m+2}{2}$.

\end{enumerate}
\end{open}

In this paper, we will settle this open problem. The rest of this paper is organized as follows. In section \ref{preliminaries}, we introduce two useful results which will
be employed in the sequel. In Section \ref{main}, we present the proof of our main result. Section \ref{conclusion} concludes the paper with some remarks.

\section{Preliminaries}\label{preliminaries}

In this section, we will introduce two useful results. The first one is about the cyclotomic coset. For a prime $p$, the $p$-cyclotomic coset modulo $p^m-1$ containing $j$ is defined as
$$\C_j=\big\{jp^s \mod(p^m-1):s=0,1,...,m-1\big\}.$$
We have the following lemma.
\begin{lemma}[Lemma 2.1, \cite{Ding-Heelseth}]\label{Lemma-size}
For any $1\leq e\leq p^m-2$ with $\gcd(e,p^m-1)=2$, the cardinality of the $p$-cyclotomic coset $\C_e$ is equal to $m$.
\end{lemma}

%\begin{lemma}\label{Lemma-irr}
%For every finite field $\gf(q)$ and every positive integer $n$, where $q$ is a prime power, the product of all monic irreducible polynomials over $\gf(q)[x]$ whose degrees divide $n$ is equal to $x^{q^n}-x$.
%\end{lemma}
It is known that a code with parameters $[3^m-1,3^m-1-2m,4]$ is optimal according to the Sphere Packing bound. To determine the optimality of $\C_{(1,e)}$, the following sufficient and necessary conditions are
given by Ding and Helleseth in \cite{Ding-Heelseth}.

\begin{theorem}[Theorem 4.1, \cite{Ding-Heelseth}]\label{thm-fundamental}
Let $e\notin\C_1$, and $|\C_e|=m$. The ternary cyclic code $\C_{(1,e)}$ has parameters  $[3^m-1,3^m-1-2m,4]$ if and only if the following conditions are satisfied:

{\textit{C1}}: e is even;

{\textit{C2}}: the equation $(x+1)^e-x^e-1=0$ has the only solution $x=0$ in $\mathsf \gf(3^m)$;

{\textit{C3}}: the equation $(x+1)^e+x^e+1=0$ has the only solution $x=1$ in $\mathsf \gf(3^m)$.

\end{theorem}

\section{Solving Open Problem \ref{open-1}}\label{main}
In this section, we confirm that each condition in Open Problem \ref{open-1} satisfies all the three conditions in Theorem \ref{thm-fundamental}. Then the answer of the open problem can be deduced. Firstly, we confirm that C1 holds in the following lemma.
\begin{lemma}\label{Lemma-size1}
Let $e=3^h+5$, where $2\leq h\leq m-1$. Then $e\notin\C_1$ and $|\C_e|=m$ if one of the following conditions is met.
\begin{enumerate}
\item $m\equiv0(\mathrm{mod}~ 4)$, $m\geq 4$, and $h=\frac{m}{2}$.

\item $m\equiv2(\mathrm{mod}~ 4)$, $m\geq 6$, and $h=\frac{m+2}{2}$.

\end{enumerate}
\end{lemma}
\begin{proof}
We only prove the first one and the second one is similar. It is easy to see that $e\notin\C_1$ since $e$ is even. It will be shown that $|\C_e|=m$.
We have
\begin{align*}
\gcd(e, 3^m-1)&=\gcd(3^{\frac{m}{2}}+5,3^m-1)=\gcd(3^{\frac{m}{2}}+5,3^m-1-(3^{\frac{m}{2}}+5)(3^{\frac{m}{2}}-5))\\
     &=\gcd(3^{\frac{m}{2}}+5,24)=\gcd(3^{\frac{m}{2}}+5,8)=(6,8)=2.
\end{align*}
The fifth equality holds since $m\equiv0(\mathrm{mod}~ 4)$ and $3^{\frac{m}{2}}+5 \equiv 6 (\mathrm{mod}~ 8)$. Consequently, $|\C_e|=m$ follows from lemma \ref{Lemma-size}.
\end{proof}
Secondly, we investigate the solutions of $(x+1)^e-x^e-1=0$ in $\gf(3^m)$.
\begin{lemma}\label{Lemma-C2}
Let $e=3^h+5$, where $2\leq h\leq m-1$. Then
\begin{equation}\label{eqnminus}(x+1)^e-x^e-1=0
\end{equation}
has the only solution $x=0$ in $\mathsf \gf(3^m)$ if one of the following conditions is met.
\begin{enumerate}
\item $m\equiv0(\mathrm{mod}~ 4)$, $m\geq 4$, and $h=\frac{m}{2}$.

\item $m\equiv2(\mathrm{mod}~ 4)$, $m\geq 6$, and $h=\frac{m+2}{2}$.
\end{enumerate}
\end{lemma}
\begin{proof}
It is obvious that $x=0$ is a solution of (\ref{eqnminus}) and $x=\pm 1$ is not. Suppose that $\theta\in \gf(3^m)\setminus \gf(3)$ and is a solution of (\ref{eqnminus}). Through a straight calculation, we have that
$$\theta^{3^h-1}(\theta^4-\theta^3-\theta^2+\theta-1)=\theta^4-\theta^3+\theta^2+\theta-1.$$
First, we assert that $\theta^4-\theta^3-\theta^2+\theta-1\neq0$. Otherwise, we have $\theta^4-\theta^3-\theta^2+\theta-1=\theta^4-\theta^3+\theta^2+\theta-1=0$, which leads to $\theta=0$. It is a contradiction. Hence we have
\begin{eqnarray}\label{eqn-solution1}
\theta^{3^{h}}=\frac{f(\theta)}{g(\theta)},
\end{eqnarray}
where $f(\theta)=\theta^5-\theta^4+\theta^3+\theta^2-\theta$ and $g(\theta)=\theta^4-\theta^3-\theta^2+\theta-1$.
Taking $3^{h}$ powers on both sides of the equation (\ref{eqn-solution1}), we have
\begin{equation}\label{eqn-solution2}
\theta^{3^{2h}}=\frac{\theta^{5\cdot3^h}-\theta^{4\cdot3^h}+\theta^{3\cdot3^h}+\theta^{2\cdot3^h}-\theta^{3^h}}{\theta^{4\cdot3^h}-\theta^{3\cdot3^h}-\theta^{2\cdot3^h}+\theta^{\cdot3^h}-1}.
\end{equation}
Plugging (\ref{eqn-solution1}) into (\ref{eqn-solution2}), we obtain
\begin{equation}\label{eqn-solution2.5}
\theta^{3^{2h}}=\frac{F(\theta)}{G(\theta)},
\end{equation}
where
$F(\theta)=f(\theta)^5-f(\theta)^4g(\theta)+f(\theta)^3g(\theta)^2+f(\theta)^2g(\theta)^3-f(\theta)g(\theta)^4$ and
$G(\theta)=f(\theta)^4g(\theta)-f(\theta)^3g(\theta)^2-f(\theta)^2g(\theta)^3+f(\theta)g(\theta)^4-g(\theta)^5$.
We distinguish the following two cases.

{\textit{Case 1}}: $m\equiv0(\mathrm{mod}~ 4)$, $m\geq 4$, and $h=\frac{m}{2}$.

Noting that $\theta^{3^{2h}}=\theta$ since $2h=m$, then (\ref{eqn-solution2.5}) becomes
$$\theta G(\theta) - F(\theta)=0.$$
With the help of Magam Program, we can decompose the left-hand side of the above equation into the product of some irreducible factors as follows.
$$\theta^3(\theta+1)(\theta-1)(\theta^6+\theta^3-\theta+1)(\theta^6-\theta^5+\theta^3+1)(\theta^6-\theta^5-\theta^3-\theta+1)=0.$$

If $\theta^6+\theta^3-\theta+1=0$, then $\theta \in \gf(3^6)\subseteq\gf(3^m)$. We have $6|m$ and then $6|h$ since $m=2h$ and $h$ is even. Plugging $\theta^{3^h}=\theta$ into the equation (\ref{eqn-solution1}), we have
$$\theta=\frac{\theta^5-\theta^4+\theta^3+\theta^2-\theta}{\theta^4-\theta^3-\theta^2+\theta-1},$$
which leads to $\theta=0$. It is a contradiction. Similarly, we can prove that $\theta^6-\theta^5+\theta^3+1\neq0$ and  $\theta^6-\theta^5-\theta^3-\theta+1\neq0$. Then $x=0$ is the only solution of (\ref{eqnminus}) in $\gf(3^m)$.

{\textit{Case 2}}: $m\equiv2(\mathrm{mod}~  4)$, $m\geq 6$, and $h=\frac{m+2}{2}$.

Noting that $\theta^{3^{2h}}=\theta^9$ since $2h=m+2$, then (\ref{eqn-solution2.5}) becomes
$$\theta^9 G(\theta) - F(\theta)=0.$$

With the help of Magam Program, we can decompose the left-hand side of the above equation into the product of some irreducible factors as follows.
$$\theta(\theta+1)(\theta-1)(\theta^4+\theta^3-\theta^2-\theta-1)(\theta^4+\theta^3+\theta^2-\theta-1)(\theta^6-\theta^5+\theta^4-\theta^3+\theta^2-\theta+1)$$
$$(\theta^8+\theta^7+\theta^6-\theta^4+\theta^2+\theta+1)(\theta^8+\theta^7-\theta^6-\theta^2+\theta+1)=0.$$

If  $\theta^6-\theta^5+\theta^4-\theta^3+\theta^2-\theta+1=0$, then $\theta \in \gf(3^6)\subseteq\gf(3^m)$ and $6|m$. It follows from $h$ is even that $h\equiv4(\mathrm{mod}~ 6)$. Noting that $\theta^7=-1$, we obtain $\theta^{3^h}=\theta^{3^4}=\theta^{77+4}=-\theta^4$. Plugging this into (\ref{eqn-solution1}), we obtain
$$\theta^6+\theta^5+\theta^4-\theta^3-\theta^2-\theta-1=0.$$
This together with $\theta^6-\theta^5+\theta^4-\theta^3+\theta^2-\theta+1=0$ leads to $\theta^5-\theta^2-1=\theta^5+\theta^3-\theta^2-1=0$. Then $\theta$ must be zero, which is a contradiction. Moreover,
$\theta^4+\theta^3-\theta^2-\theta-1, \theta^4+\theta^3+\theta^2-\theta-1, \theta^8+\theta^7+\theta^6-\theta^4+\theta^2+\theta+1$ and $\theta^8+\theta^7-\theta^6-\theta^2+\theta+1$ cannot be zero since $4\nmid m$. Then $x=0$ is the only solution of (\ref{eqnminus}) in $\gf(3^m)$.
This completes the proof.
\end{proof}
In what follows, we investigate the solutions of $(x+1)^e+x^e+1=0$ in $\gf(3^m)$.
\begin{lemma}\label{Lemma-C3}
Let $e=3^h+5$, where $2\leq h\leq m-1$. Then
\begin{equation}\label{eqnplus}
(x+1)^e+x^e+1=0
\end{equation} has the only solution $x=1$ in $\mathsf \gf(3^m)$ if one of the following conditions is met.
\begin{enumerate}
\item $m\equiv0(\mathrm{mod}~ 4)$, $m\geq 4$, and $h=\frac{m}{2}$.

\item $m\equiv2(\mathrm{mod}~ 4)$, $m\geq 6$, and $h=\frac{m+2}{2}$.
\end{enumerate}
\end{lemma}
\begin{proof}
It is obvious that $x=1$ is a solution of (\ref{eqnplus}). Suppose that $\theta\in\gf(3^m)\setminus \gf(3)$ is a solution of (\ref{eqnplus}). Through a straight calculation, we have
$$\theta^{3^h}(\theta^4-\theta^3+\theta^2+1)=\theta^4+\theta^2-\theta+1.$$
First, we assert that $\theta^4-\theta^3+\theta^2+1\neq0$. Otherwise, we have $\theta^4-\theta^3+\theta^2+1=\theta^4+\theta^2-\theta+1=0$. It then follows that $\theta^3=\theta$. This is contrary to the assumption that $\theta \notin \gf(3)$. Hence, we have
\begin{eqnarray}\label{eqn-solution3}
\theta^{3^{h}}=\frac{k(\theta)}{l(\theta)},
\end{eqnarray}
where $k(\theta)=\theta^4+\theta^2-\theta+1$ and $l(\theta)=\theta^4-\theta^3+\theta^2+1$. Taking $3^{h}$ powers on both sides of the equation (\ref{eqn-solution3}), we have
\begin{equation}\label{eqn-solution4}
\theta^{3^{2h}}=\frac{\theta^{4\cdot3^h}+\theta^{2\cdot3^h}-\theta^{3^h}+1}{\theta^{4\cdot3^h}-\theta^{3\cdot3^h}+\theta^{2\cdot3^h}+1},
\end{equation}
Plugging (\ref{eqn-solution3}) into (\ref{eqn-solution4}), we obtain
\[\theta^{3^{2h}}=\frac{K(\theta)}{L(\theta)},\]
where $K(\theta)=k(\theta)^4+k(\theta)^2l(\theta)^2-k(\theta)l(\theta)^3+l(\theta)^4$ and $L(\theta)=k(\theta)^4-k(\theta)^3l(\theta)+k(\theta)^2l(\theta)^2+l(\theta)^4$. We distinguish the following two cases.

{\textit{Case 1}}: $m\equiv0(\mathrm{mod}~ 4)$, $m\geq 4$, and $h=\frac{m}{2}$.

Noting that $\theta^{3^{2h}}=\theta$ since $2h=m$, then $\theta$ satisfies
$$\theta L(\theta)-K(\theta)=0.$$
With the help of Magam Program, we can decompose the left-hand side of the above equation into the product of some irreducible factors as follows.
$$(\theta-1)^5(\theta^2+\theta-1)^2(\theta^2-\theta-1)^2(\theta^2+1)^2=0.$$

If $\theta^2+\theta-1=0$, then $\theta \in \gf(3^2)$. We have $\theta^{3^2}=\theta$ and then $\theta^{3^h}=\theta$ since $h$ is even. Plugging $\theta^{3^h}=\theta$ into (\ref{eqn-solution3}), then we have
$$\theta=\frac{\theta^4+\theta^2-\theta+1}{\theta^4-\theta^3+\theta^2+1},$$
which leads to $\theta^5+\theta^4+\theta^3-\theta^2-\theta-1=(\theta-1)^5=0$, a contradiction. Similarly, we can prove that $\theta^2+\theta-1\neq0$ and $\theta^2+1\neq0$. Then $x=1$ is the only solution of (\ref{eqnplus}) in $\gf(3^m)$.

{\textit{Case 2}}: $m\equiv2(\mathrm{mod}~ 4)$, $m\geq 6$, and $h=\frac{m+2}{2}$.

Noting that $\theta^{3^{2h}}=\theta^9$ since $2h=m+2$, then $\theta$ satisfies
$$\theta^9 L(\theta)-K(\theta)=0.$$
With the help of Magam Program, we can decompose the left-hand side of the above equation into the product of some irreducible factors as follows.
$$(\theta-1)(\theta^2+1)(\theta^2+\theta-1)(\theta^2-\theta-1)(\theta^3-\theta+1)(\theta^3-\theta-1)$$
\[(\theta^3+\theta^2-\theta+1)(\theta^3-\theta^2+\theta+1)(\theta^3+\theta^2-1)(\theta^3-\theta^2+1)=0.\]
Similar with the proof of Case 1, we know that $\theta^2+1,\theta^2+\theta-1,\theta^2-\theta-1\neq0$. If $\theta^3-\theta-1=0$, then $\theta \in \gf(3^3)\subseteq \gf(3^m)$. We have $\theta^{3^3}=\theta$ and $3|m=2h-2$, this leads to $h\equiv1(\mathrm{mod}~ 3)$ and $\theta^{3^h}=\theta^3$. Plugging $\theta^{3^h}=\theta^3$ into (\ref{eqn-solution3}), we have
$$\theta^3=\frac{\theta^4+\theta^2-\theta+1}{\theta^4-\theta^3+\theta^2+1},$$
which leads to $\theta^8=1$. It follows from $\theta^{3^3-1}=1$ that $\theta^2=\theta^{(8,3^3-1)}=1$, a contradiction. This completes the proof of Case $2$.
\end{proof}

The answer to Open Problem \ref{open-1} is given in the following theorem.

\begin{theorem}
Let $e=3^h+5$, where $2\leq h\leq m-1$. Let $m$ be even. Then the ternary cyclic code $\C_{(1,e)}$ has parameters  $[3^m-1,3^m-1-2m,4]$ if one of the following conditions is met.
\begin{enumerate}
\item $m\equiv0(\mathrm{mod}~ 4)$, $m\geq 4$, and $h=\frac{m}{2}$.

\item $m\equiv2(\mathrm{mod}~ 4)$, $m\geq 6$, and $h=\frac{m+2}{2}$.

\end{enumerate}
\end{theorem}
\begin{proof}
The conclusions follow from Lemma \ref{Lemma-size1}, Lemma \ref{Lemma-C2}, Lemma \ref{Lemma-C3} and Theorem \ref{thm-fundamental}.
\end{proof}

\section{Conclusions}\label{conclusion}

In this paper, we settled an open problem proposed by Ding and Helleseth in 2013 about a class of optimal ternary cyclic codes. The main technique we used is shown in solving the equation in conditions C2 and C3. Assume that $\theta$ is a solution of the target equation, we can obtain $\theta^{3^h}=R(\theta)$ after calculation, where $R(\theta)$ is a rational function of $\theta$ with known degree and coefficients. Then we take $3^h$-th power of $\theta^{3^h}=R(\theta)$, together with the relationship between $m$ and $h$, we can find an solvable equation of $\theta$. We remark that when $h$ is close to $\gamma m$, where $\gamma$ is a rational number, our technique always works. For instance, the following theorem gives other optimal cyclic codes with respect to the Sphere Packing bound. This gives an incomplete answer to Open Problems 7.12-7.15 in \cite{Ding-Heelseth}.
\begin{theorem}
Let $m$ be an odd integer no less than five and $\gcd(m,3)=1$. Then the ternary cyclic code $\C_{(1,e)}$ has parameters  $[3^m-1,3^m-1-2m,4]$ if one of the following conditions is met.
\begin{enumerate}
\item $e=3^h+5$, where $2h\equiv\pm1(\mathrm{mod}~ m)$;

\item $e=3^h+13$, where $2h\equiv\pm1(\mathrm{mod}~ m)$;

\item $e=\frac{3^{m-1}}{2}+3^h+1$, where $2h\equiv\pm1(\mathrm{mod}~ m)$ or $3h\equiv\pm1(\mathrm{mod}~ m)$ or $4h\equiv\pm1(\mathrm{mod}~ m)$.
\end{enumerate}
\end{theorem}

%\section*{Acknowledgments}
%The authors are very grateful to the reviewer and the
%Associate Editor, Prof. James W.P. Hirschfeld, for their comments and
%suggestions that improved the presentation and quality of this
%paper. This work was finished when the authors visited the Hong Kong University of Science and Technology. The authors
%are grateful to Professor Cunsheng Ding for bringing them together
%in the summer of 2014.

%\section*{Acknowledgements}
%The authors are grateful to the anonymous reviewers for careful reading and for invaluable suggestions which improve the quality of the paper.

\end{document}